\documentclass{amsart}[12pt,a4paper]

\usepackage[applemac]{inputenc}
\usepackage[colorlinks=true,pagebackref,hyperindex]{hyperref}
\usepackage[english]{babel} 
\usepackage{latexsym}
\usepackage[T1]{fontenc} 
\usepackage{graphicx}
\usepackage{graphics} 
\usepackage{fancyhdr}
\usepackage{enumerate}
\usepackage{color} 

\usepackage[a4paper,margin = 3.3cm ]{geometry}

\usepackage{amsfonts}  
\usepackage{amsmath}
\usepackage{amssymb} 
\usepackage[alphabetic]{amsrefs}
\usepackage{amsthm}
\usepackage{mathrsfs} 
\input xy
\xyoption{all}

\newtheorem{theorem}{Theorem}[section]

\newtheorem{example}[theorem]{Example}

\newtheorem{definition-lemma}[theorem]{Definition-Lemma}

\newtheorem{defn}[theorem]{Definition}
\newtheorem{remark}[theorem]{Remark}

\newtheorem*{conconj}{Cone conjecture}
\newtheorem*{definv}{Deformation invariance of plurigenera}
\newtheorem*{canonical bundle}{Canonical bundle formula}

\title{Invariance of Plurigenera and boundedness for Generalized  Pairs}
\author{Stefano Filipazzi, Roberto Svaldi}

\newcommand{\Q}{\mathbb{Q}}			
\newcommand{\R}{\mathbb{R}}			
\newcommand{\C}{\mathbb{C}}			
\newcommand{\N}{\mathbb{N}}			

\newcommand{\rar}{\rightarrow}		
\newcommand{\drar}{\dashrightarrow}		
\newcommand{\Xf}{\mathcal{X}}		
\newcommand{\Yf}{\mathcal{Y}}		
\newcommand{\Wf}{\mathcal{W}}		
\newcommand{\Bf}{\mathcal{B}}		
\newcommand{\Mf}{\mathcal{M}}		
\newcommand{\Df}{\mathcal{D}}		
\newcommand{\Af}{\mathcal{A}}		
\newcommand{\Hf}{\mathcal{H}}		

\DeclareMathOperator{\Proj}{Proj}		
\DeclareMathOperator{\Supp}{Supp}		
\DeclareMathOperator{\vol}{vol}		
\DeclareMathOperator{\rk}{rk}			
\DeclareMathOperator{\coeff}{coeff}	
\DeclareMathOperator{\coloneqq}{:=}

\def\O#1.{\mathcal {O}_{#1}}			
\def\pr #1.{\mathbb P^{#1}}				
\def\af #1.{\mathbb A^{#1}}				
\def\ses#1.#2.#3.{0\to #1\to #2\to #3 \to 0}		
\def\xrar#1.{\xrightarrow{#1}}			
\def\K#1.{K_{#1}}						
\def\bM#1.#2.{\mathbf{M}_{#1,#2}}				
\def\bL#1.#2.{\mathbf{L}_{#1,#2}}				

\def\subs#1.{_{#1}}						
\def\sups#1.{^{#1}}						

\DeclareMathOperator{\Diff}{Diff}		

\usepackage{lipsum}

\address{S.~Filipazzi, \textsc{UCLA Mathematics Department,
Box 951555,
Los Angeles, CA 90095-1555, USA}} \email{filipazzi@math.ucla.edu}

\address{R.~Svaldi, \textsc{EPFL, SB MATH-GE, MA B1 497 (B\^{a}timent MA), Station 8, CH-1015 Lausanne, Switzerland}}
\email{roberto.svaldi@epfl.ch}



\begin{document}
\setcounter{tocdepth}{1}

\begin{abstract}
In this note, we survey some recent developments in birational geometry concerning the boundedness of algebraic varieties.
We delineate a strategy to extend some of these results to the case of generalized pairs, first introduced by Birkar and Zhang, when the associated log canonical divisor is ample, and the volume is fixed.
In this context, we show a version of deformation invariance of plurigenera for generalized pairs.
We conclude by discussing an application to the boundedness of varieties of Kodaira dimension $\kappa(X)=\dim(X)-1$.
\end{abstract}

\thanks{SF was supported by a Graduate Research Fellowship awarded by the University of Utah. Partial support was also provided by NSF research grants no: DMS-1300750, DMS-1265285 and by a grant from the Simons Foundation; Award Number: 256202.
RS was partially supported by funding from the European Union's Seventh Framework Programme (FP7/2007-2013)/ERC Grant agreement no. 307119, which also provided support for the visit of SF to SISSA, Trieste and by Churchill College, Cambridge.
He was also partially supported from the European Union's Horizon 2020 research and innovation programme under the Marie Sk\l{}odowska-Curie grant agreement No. 842071.}

\maketitle

\tableofcontents

\section{Introduction}
Throughout this paper, we work over an algebraically closed field of characteristic 0, for instance, the complex number field, $\mathbb{C}$.

One of the main goals in algebraic geometry is to realize a sufficiently synthetic albeit complete classification of projective varieties, that is, subsets of projective space defined by the vanishing of finitely many homogeneous polynomials.
To this end, there are two possible distinct approaches: either by identifying two distinct varieties if they are isomorphic or by introducing the notion of birational equivalence. Two algebraic varieties are birationally equivalent (or simply, birational) when they both contain isomorphic dense open sets.
When algebraic varieties are birational, many numerical and geometrical quantities that capture their structure are preserved. 
Hence, birational equivalence is a sufficiently coarse equivalence relation among geometrical objects. 
At the same time, it allows more flexibility than just the classification by isomorphism type: we are free to modify the variety under scrutiny as long as a dense open set is left untouched; the new variety thus obtained is birational to the original one.
Indeed, this is the leitmotif of the whole birational classification: among all varieties in a given birational class, we would like to find one whose geometric features are the best possible.
Of course, part of the problem is to make sense of what the expression ``best possible'' means in the previous sentence.

A very important role in this task is played by the canonical bundle of a normal variety.
For a smooth variety, that is just defined as the determinant of the cotangent bundle.
In the singular case, normality implies that the smooth locus has a complement of codimension at least two within the variety; thus, we can extend the canonical bundle from the smooth locus to the whole variety in a natural way -- although it will no longer be a line bundle, but rather a Weil divisorial sheaf.

Starting with Mori in the 1980s and then continuing with many other important contributors of birational geometry up until this very day, it is very well understood that one way to construct a preferred representative in the birational equivalence class of a projective variety can be achieved by making the canonical divisor ``as positive as possible''.
In more precise terms, this means that one would like to find a birational model of a given projective variety on which the canonical divisor becomes numerically effective.
That birational model is then called a {\it minimal model}.
To be able to construct minimal models, it is inevitable to consider singular varieties within a birational equivalence class, as it is already clear in dimension three -- unlike the case of surfaces.
Nonetheless, it is enough to consider a well-behaved class of singularities, which has now been intensively studied, cf. \S \ref{sect.bound}, \ref{g.pair.sect}.

One of the main open problems in birational geometry is whether minimal models do exist.
Indeed, they are conjectured to exist if and only if the varieties within a given birational equivalence class are not covered by rational curves.
Such varieties are said to be uniruled.
A series of conjectures, known as the Minimal Model Program, predicts that minimal model exists for non-uniruled varieties with mild singularities and moreover provides a conjectural algorithmic construction for them.
More generally, the Minimal Model Program predicts that, up to some special birational equivalences, each projective variety decomposes into iterated fibrations with general fibers of $3$ basic types:
\begin{itemize}
	\item {\it log Fano varieties}: varieties with ample anti-canonical bundle;
	\item {\it $K$-trivial varieties}: varieties with torsion canonical bundle; and
	\item {\it log canonical models}: varieties with ample canonical bundle.
\end{itemize}  
The classification scheme then proceeds with the study of the geometry of these three special types of varieties.
In particular, under the perspective of the minimal model program, the classification process can be further subdivided into two main goals:
\begin{enumerate}
\item the construction of moduli spaces for varieties in each of the three key types just introduced; and
\item the study of the structure of these moduli spaces.
\end{enumerate}

In particular, the latter task should be thought in connection with the study of fibrations whose general fibers fall into one of the three fundamental types above.
In fact, given a fibration $f \colon X \to Y$ where the general fiber is either one of the three basic types introduced above, assuming the existence of a moduli space $\mathcal{M}$ parametrizing the isomorphism types of the generic geometric fibers, then by the definition of a moduli functor there is an induced rational map $Y \dashrightarrow \mathcal{M}$ (or rather a rational map to the coarse moduli space of $\mathcal{M}$), associating to a sufficiently general point $y \in Y$ the class of isomorphism of the fibre $X_y$.
Hence, knowing the structure of the moduli space can help us understand the structure of the fibration $f$. 

The process of constructing moduli spaces for a given class of algebraic varieties has several steps.
The first step is to show that the chosen class of varieties is bounded, i.e., it can be parametrized by a finite number of parameters.
For instance, if we look at smooth projective curves, once we fix the genus $g \geq 2$, it has been known since Riemann that these vary in a $(3g-3)$-dimensional family.
In this case, it is easy to see that the bi-canonical linear system provides the desired embedding.
\newline
Once boundedness is settled, the next step is to find a functorial construction for a parameter space.
As it is often easier to work with compact (or projective) varieties, we would like our parameter space to be compact.
On the other hand, we would like that the extra points needed to obtain a compact parameter space were related to our original classification problem -- that is, we would like to define a functor whose moduli space is proper.
The new points should represent the limit of well-behaved degenerations of families of varieties in the chosen class.
This whole circle of ideas leads to the construction of a moduli functor and eventually of a moduli space.
Deligne and Mumford, \cite{MR0262240}, showed that a moduli space of curves of genus $g \geq 2$ exists and it can be naturally compactified by considering so-called stable curves, nodal curves with ample canonical class.

In this note, we survey some of the recent techniques and results that have emerged in very recent years in relation to the study of boundedness for algebraic varieties.   
Moreover, we explain a possible attempt at extending these results to the class of generalized pairs, cf. \S \ref{g.pair.sect}, that was recently introduced by Birkar and Zhang \cite{BZ}. 
A result of this type would, for example, provide boundedness for the images of the Iitaka fibrations of varieties of intermediate Kodaira dimension.
As a propaedeutic step, we show that the dimensions of the spaces of sections of positive multiples of the log divisors associated with generalized pairs are constant in families, see Theorem \ref{deformation invariance}.
This is a crucial step in the completion of the plan that we detail for the boundedness of ample generalized pairs with fixed volume.

The structure of the paper is as follows: 
in Section \ref{sect.bound} we introduce the formal definition of boundedness, and we illustrate some of the recent progress on the problem, as well as some of the open challenges; 
in Section \ref{g.pair.sect}, we discuss the notion of generalized pair and explain how that plays an important role in boundedness problems for minimal models; 
Section \ref{def.inv.sect} is devoted to the proof of the invariance of plurigenera for big and nef klt generalized pairs; 
finally, in Section \ref{fibr.sect}, we show a boundedness result of birational type for elliptic fibrations, and we discuss its relation to a famous conjecture of Kawamata and Morrison.

\subsection*{Acknowledgements}
RS would like to thank the organizers of the Conference ``Moduli spaces in Algebraic Geometry and Applications'' for the opportunity to speak in Campinas and the extremely enjoyable and productive environment that they were able to create for the occasion.
\newline
The authors would like to thank Caucher Birkar, Paolo Cascini, Gabriele Di Cerbo, Christopher Hacon, James M\textsuperscript{c}Kernan, Luca Tasin, and Chenyang Xu for many useful conversations on the topics of this work over the years.
Part of this work was completed during two visits of SF to RS at SISSA, Trieste and the University of Cambridge.
SF would like to thank both institutions for the hospitality and the nice working environment.

\section{A tour of boundedness} \label{sect.bound}
\subsection*{Boundedness}
When we consider a set $\lbrace X_i \rbrace _{i \in I}$ of varieties, the first step towards constructing a well-behaved parameter space is making sure that they can all be embedded in the same projective space $\pr N.$ in a controlled way.
The theory of Hilbert schemes suggests that, if there are only finitely many possible Hilbert polynomials for the $X_i$ with respect to such an embedding into $\pr N.$, then the $X_i$ will naturally be the fibers of a family of varieties parametrized by a scheme of finite type.
The notion of boundedness is simply a generalization of this idea.

\begin{defn} \label{def boundedness}
{\em
A set of projective varieties $\lbrace X_i \rbrace _{i \in I}$ is said to be {\it bounded} if there exists a projective morphism of algebraic varieties $\mathcal{X} \rar T$, where $T$ is of finite type, such that for any $X \in \lbrace X_i \rbrace _{i \in I}$ there exists a closed point $t \in T$ for which the fiber $\mathcal{X}_t$ is isomorphic to $X$.
}
\end{defn}
When a set of varieties is bounded, we should expect that, upon partitioning them into finitely many subsets, they share many geometric features.
For example, if all of the $X_i$ are smooth and of the same dimension, then they only have finitely many possible distinct underlying topological spaces, as implied by Ehresmann's theorem, see \cite{Voisin}*{Theorem 9.3}.

As we work with reduced and irreducible schemes, if we fix dimension and degree of subvarieties of $\pr N.$, the theory of Chow varieties guarantees that they form a bounded family in the sense of Definition \ref{def boundedness}, see \cite[\S 1.3]{k.book96}.
Thus, one general strategy to prove that a set of varieties $\{ X_i\}_{i\in I}$ of fixed dimension $d$ is bounded is to find a very ample line bundle on each $X_i$ that embeds it with degree bounded from above in a projective space of bounded dimension.
This is a first hint to the fact that, when we want to construct moduli spaces or more generally address boundedness questions regarding algebraic varieties, we need to fix some invariants.
We have already discussed the case of curves in the introduction: there, it suffices to fix the genus $g$ of a smooth projective curve in order to construct a good moduli functor with proper moduli space $\overline{\mathcal{M}}_g$.
The genus $g$ is a topological invariant of smooth projective curves, but it can also be readily read off from the degree of the cotangent bundle $\mathcal{O}_C(K_C)$ of a curve $C$: $\deg \mathcal{O}_C(K_C) = 2g(C)-2$.
As the linear system $|2K_C|$ embeds $C$ in $\mathbb{P}(H^0(C, \mathcal{O}_C(2K_C))^\vee) \simeq \mathbb{P}^{3g-4}$, we have reproven the boundedness of smooth curves of fixed genus.
\subsection*{Volume}
Recall that the \emph{volume} of a Cartier divisor $D$ on a projective variety $Y$ is defined as 
\[
\vol(Y,D) \coloneqq \limsup_{m \to \infty} \frac{h^0(Y,\O Y.(mD))}{m^n/n!},
\]
where $n = \dim (Y)$. If $D$ is a $\Q$-Cartier $\Q$-divisor, we set $\vol(Y,D) \coloneqq \frac{\vol(Y,kD)}{k^n}$, where $kD$ is Cartier.
Hence, in the case of curves, as we have $\vol(C,\K C.)=2 g(C)-2$, we can think of $\mathcal{M}_g$ as being obtained by putting a constraint on the volume of the canonical divisor.
Unlike the topological genus of a Riemann surface, the perspective given by the volume is suitable for a generalization.
In particular, starting from a birational viewpoint, we may consider smooth $n$-dimensional projective varieties of general type, that is, varieties with the property that the volume of the canonical bundle is positive.
An equivalent characterization is given by requiring that the Iitaka fibration, cf. \cite[Theorem 2.1.33]{LAZ1}, is a birational map.
The expectation is that general type varieties provide the generalization in the birational world of varieties with ample canonical bundle.
We will explain below how this intuition is actually well-rooted in results from the Minimal Model Program.
Much in the same vein, we could do something similar for smooth $n$-dimensional Fano or $K$-trivial varieties and their birational equivalence classes.
For the purpose of this note, though, we will only focus on the general type case.

Unlike the case of curves, higher dimensional varieties have interesting birational geometry.
Already by blowing up smooth points on surfaces, we realize that fixing the dimension $n$ and the volume $v$ for the canonical divisor is not enough to obtain a quasi-projective parameter space.
Indeed, while the isomorphism type of a curve is the same as the birational equivalence type, in dimension at least 2, any birational equivalence class contains infinitely many non-isomorphic varieties.
For varieties of general type, this is reflected in the fact that the canonical bundle provides just a birational polarization: the condition that the volume of the canonical divisor is positive is much weaker than requiring it to be ample.
On the other hand, we cannot hope to find a smooth birational model $X'$ of a smooth general type variety $X$ with $\K X'.$ ample; this is already evident for surfaces of general type, where we encounter ADE singularities when attempting to construct the canonical model, cf. \cite[Chapter 4]{KM}. 

Rather than regarding a rich birational geometry and the presence of singularities as a problem, we can try to take advantage of the flexibility that these provide.
In particular, we can introduce weaker notions of boundedness that work for any variety in a given birational equivalence class.

\begin{defn} \label{def b boundedness}
{\em
A set of projective varieties $\lbrace X_i \rbrace _{i \in I}$ is said to be {\it birationally bounded} if there exists a projective morphism of algebraic varieties $\mathcal{X} \rar T$, where $T$ is of finite type, such that for any $X \in \lbrace X_i \rbrace _{i \in I}$ there exists a closed point $t \in T$ for which the fiber $\mathcal{X}_t$ is birationally equivalent to $X$.
}
\end{defn}

We have already discussed how, in order to construct moduli spaces, we often have to fix some numerical invariants within a given class of projective varieties.
For smooth varieties of general type, the two most natural invariants to fix are the dimension $n$ together with the volume of the canonical bundle $v$.
Once these invariants are specified, we can ask whether the varieties satisfying these constraints are birationally bounded.
It turns out that fixing the dimension $n$ and the volume $v=\vol(X,\K X.)$ is enough to achieve birational boundedness of smooth projective varieties of general type, see \cite[Corollary 1.2]{HM06}.
Roughly speaking, fixing the volume guarantees that a fixed multiple $|m \K X.|$ defines a birational map to a variety embedded into a fixed projective space $\pr N.$.
Furthermore, the bound on $\vol(X,\K X.)$ also gives a bound on the degree of the image of this birational map.

Once birational boundedness is achieved, it is natural to wonder whether there is a natural representative in each birational class of varieties of general type for which the canonical bundle is ample.
More precisely, can we choose one specific such representative within each birational class of varieties of general type and achieve boundedness for these models?
If we do not want to leave the realm of smooth varieties, we have already seen that this question has a positive answer just up to dimension 2.
On the other hand, if we are willing to admit varieties with mild singularities, the Minimal Model Program provides us with a positive answer in any dimension.
More precisely, if $X$ is smooth with $\vol(X,\K X.) > 0$, there exists a birational contraction $X \drar X'$ such that $X'$ has canonical singularities, $\K X'.$ is ample and $\vol(X,\K X.) = \vol(X',\K X'.)$.
The variety $X'$ is called the \emph{canonical model}.
It is unique and is characterized as $X' = \Proj(\bigoplus_{m \geq 0} H^0(X,m\K X.))$, cf. \cite{BCHM}.

\subsection*{Log pairs}
If we adopt the perspective of the Minimal Model Program, we can inquire boundedness in broader generality.
In the context of the classification, it is often more convenient to work with a slightly more general type of objects, namely, {\it log pairs} (or simply {\it pairs} for short).
A pair $(X, \Delta)$ consists of a normal variety $X$ and an effective $\R$-divisor $\Delta$ with coefficients in $(0,1]$ on $X$ such that $K_X+\Delta$ is $\mathbb{R}$-Cartier.
Such pairs appear quite naturally when generalizing the adjunction formula to singular varieties: when $X$ is a mildly singular hypersurface in a mildly singular variety $Y$, then the classical adjunction formula $(K_Y+X)|_X = K_X$ often fails to hold. 
One then needs a correction term in the form of an effective divisor, that is, the adjunction formula looks like $(K_Y+X)|_X=K_X+\Delta$, for some $\Delta \geq 0$ on $X$.
Given a log resolution $f\colon X' \rightarrow X$ of the log pair $(X, \Delta)$, we write
\[
\K X'. +\Delta' =f^*(K_X+\Delta),
\]
where $\Delta'$ is the unique divisor for which $f_\ast (\K X'.+\Delta') = K_X+\Delta$.
Thus, $\Delta'$ is the sum of the strict transform $f^{-1}_\ast \Delta$ of $\Delta$ on $X'$ and a divisor completely supported on the exceptional locus of $f$. 
Denoting by $\mu_D (\Delta')$ the multiplicity of $\Delta'$ along a prime divisor $D$ on $X'$, for a non-negative real number $\epsilon$, the log pair $(X,\Delta)$ is called
\begin{itemize}
\item[(a)] \emph{$\epsilon$-Kawamata log terminal} (\emph{$\epsilon$-klt}, in short) if $\mu_D (\Delta')< 1-\epsilon$ for all $D \subset X'$;
\item[(b)] \emph{$\epsilon$-log canonical} (\emph{$\epsilon$-lc}, in short) if $\mu_D (\Delta')\leq  1-\epsilon$ for all $D \subset X'$;
\item[(c)] \emph{terminal} if  $\mu_D (\Delta')< 0$ for all $f$-exceptional $D \subset X'$ and all possible choices of $f$;
\item[(d)] \emph{canonical} if  $\mu_D (\Delta')\leq 0$ for all $f$-exceptional $D \subset X'$ and all possible choices of $f$.
\end{itemize}

The case $0$-lc (respectively $0$-klt) case coincides with canonical (resp. terminal) singularities, and we omit it from the notation.
We can extend the discussion of the previous subsection to the case of pairs.
More precisely, we can consider log canonical pairs $(X,\Delta)$ of log general type, that is, $\vol(X, K_X+\Delta) >0$. 
We may try to fix certain numerical invariants to determine whether such a class of pairs is bounded.
Again, a natural choice of invariants to fix is $\dim (X)$ and the log canonical volume $\vol(X,\K X. + \Delta)$. 
Nonetheless, we also need to put some technical (yet natural) constraints on the possible coefficients of $\Delta$.
Once these are fixed, we can ask whether these pairs are birationally bounded.
Let us notice that, when we talk about the boundedness of pairs, we require that the supports of the boundaries deform in the bounding family, cf. \cite[2.1 Notations and Conventions]{HMX18}.

\subsection*{Boundedness for varieties of general type}
Our main reason to introduce singular varieties and pairs is that singularities are unavoidable when running the Minimal Model Program in order to realize (log) canonical models.
If our initial input is a smooth variety (respectively, a klt pair, a log canonical pair), the canonical model (resp. log canonical model) is a canonical variety (resp., a klt pair, a log canonical pair).
On the other hand, already in the case of algebraic curves, non-normal degenerations are needed to compactify $\mathcal{M}_g$ in a modular way and obtain $\overline{\mathcal{M}}_g$.
In higher dimension, the correct generalization of this notion is given by so-called semi-log canonical pairs.
Roughly speaking, semi-log canonical pairs are the generalization in higher dimension of stable pointed curves, and it is natural to address boundedness of these, see \cite{koll.mod}.

In this generality, Hacon, M\textsuperscript{c}Kernan and Xu have proved the following boundedness result.
\begin{theorem} 
\cite[Theorem 1.2.1]{HMX18} 
\label{boundedness hmx}
Fix $n \in \N$, $d > 0$ and a DCC set $I \subset [0,1] \cap \Q$.
Then, the set $\mathfrak{F}_{\rm slc}(n,I,d)$ of pair $(X, \Delta)$ such that
\begin{enumerate}
    \item $(X,\Delta)$ is a semi-log canonical pair,
    \item $\dim( X) =n$, 
    \item $\K X. + \Delta$ is ample,
    \item $\vol(X,\K X. + \Delta)=d$, and 
    \item $\coeff(\Delta) \subset I$,
    \end{enumerate}
 is bounded.
\end{theorem}

Let us highlight some of the main ideas in the proof of Theorem \ref{boundedness hmx}:
one can reduce from the case of semi-log canonical pairs to that of log canonical pairs, thanks to Koll\'ar's gluing theory, cf. \cite[Chapter 5]{Kol13}, and to another deep result of Hacon, M\textsuperscript{c}Kernan, and Xu, who proved a structure theorem for the possible volumes $\vol(X,\K X. + \Delta)$ \cite[Theorem 1.3]{HMX14b}: indeed, they show that such set satisfies the Descending Chain Condition (in short, DCC), i.e., any descending sequence with values in the set is eventually constant.
This last result crucially relies on the coefficients of the boundary $\Delta$ satisfying in turn the DCC. 
Once we can reduce to the lc case, as $\vol(X,\K X. \Delta)= v$ is fixed, we obtain a birationally bounded family $(\mathcal{X},\mathcal{B}) \rar T$, in which the lc pairs $(X, \Delta)$ satisfying the conditions of the theorem fit, up to birational isomorphism.
In order to conclude, one would like to run a suitable Minimal Model Program $\mathcal{X} \drar \mathcal{X}'$ over $T$ to obtain a family of log canonical models, i.e., exactly those pairs for which we wish to show boundedness.
Hacon, M\textsuperscript{c}Kernan, and Xu showed that this indeed holds.
One of the key ingredients in their strategy is the deformation invariance of the plurigenera $h^0(\mathcal{X}_t,m(\K \mathcal{X}_t. + \mathcal{B}_t))$, which guarantees that the aforementioned Minimal Model Program $\mathcal X \dashrightarrow \mathcal{X}'$ preserves the pluricanonical ring fiber by fiber.

\section{The canonical bundle formula and generalized pairs}\label{g.pair.sect}

\subsection*{Varieties of intermediate Kodaira dimension} The Minimal Model Program predicts that every variety can be birationally decomposed as iterated fibrations of three fundamental types of varieties: varieties of general type, $K$-trivial varieties and Fano-type varieties.
A similar phenomenon is predicted in the case of pairs.
Therefore, in order to address boundedness questions about more complicated classes of varieties, it is necessary to settle the boundedness of the three key building blocks.

The work of Hacon, M\textsuperscript{c}Kernan, and Xu establishes boundedness results for varieties of general type, while that of Birkar does the same in the Fano-type case \cites{Bir16a,Bir16b}.
Some recent results were also obtained in the case of $K$-trivial varieties, cf. \cites{dCS17, CDHJS, BdCS}.
In between varieties of general type and $\K .$-trivial ones, we have varieties of intermediate Kodaira dimension.
More precisely, we have varieties $X$ for which $h^0(X,m\K X.)$ admits an asymptotic estimate as $C_1 m \leq h^0(X,m \K X.) \leq C_2 m^{\dim (X) - 1}$ for $m$ large and divisible.

Under the perspective of the Minimal Model Program, we can regard varieties of intermediate Kodaira dimension as fibrations of $K$-trivial varieties over bases of general type.
This decomposition goes as follows.
Let $X$ be one of these varieties, and assume it has klt singularities.
For simplicity, assume that $\K X.$ is semi-ample.
This is a natural assumption in birational geometry, as it is conjectured that every klt variety $Y$ of non-negative Kodaira dimension admits a birational contraction $Y \drar Y'$ such that $\K Y'.$ is semi-ample \cite[Conjecture 2.8, Conjecture 5.7]{HM10}.
Then, as $|l \K X.|$ is basepoint-free for some $l \gg 0$, we have a naturally induced morphism $f \colon X \rar Z$, the so called \emph{Iitaka fibration}, to a normal projective variety $Z$.
By construction, we have $Z = \mathrm{Proj}(\bigoplus \subs m \geq 0. H^0(X,\O X.(m\K X.)))$ and $\K X. \sim \subs \Q. f^*L_Z$, where $L_Z$ is an ample $\Q$-Cartier divisor on $Z$.
By repeated adjunction, we have $\K X_z. = \K X. | \subs X_z.$, where $X_z$ is a general fiber of $f$.
In particular, we have that $\K X_z. \sim \subs \Q. 0$.
Thus, the general fibers of $f$ are $K$-trivial varieties.
On the other hand, it is a priori unclear how to regard $Z$ as a variety of general type, since $\K Z.$ may not be big in general.
The \emph{canonical bundle formula}, as discussed below in Remark \ref{remark cbf}, provides the right perspective on this phenomenon.
Indeed, we can (almost) canonically find an effective divisor $\Delta_Z$ such that $(Z,\Delta_Z)$ is klt and $\K Z. + \Delta_Z \sim_\Q L_Z$.
Since $L_Z$ is ample, then $(Z,\Delta_Z)$ is a pair of general type.

\subsection*{The canonical bundle formula}
Let $(X,B)$ be a projective klt pair, and let $f \colon X \rar Z$ be a morphism with connected fibers.
Assume there is a $\Q$-Cartier divisor $L_Z$ on $Z$ such that $\K X. + B \sim_\Q f^*L_Z$.
As in the case of the Iitaka fibration discussed above, the general fiber $(X_z,B_z)$ of $f$ is a $K$-trivial pair.
A special instance of this setup is the case of a minimal elliptic surface $g \colon S \rar C$, where the general fiber is an elliptic curve and $\K S. \sim g^* L_C$ for some Cartier divisor $L_C$ on the curve $C$.
In this case, Kodaira showed that one can write $L_C \sim_\Q K_C + B_C+ M_C$, where $B_C$ is a divisor measuring the singular fibers of $g$, and $M_C$ is measuring the variation of the smooth fibers \cite{Kod63}.
More precisely, $B_C$ can be explicitly computed from the multiplicities and dual graphs of the singular fibers, and $M_C= \frac{1}{12}j^* \mathcal{O}_{\pr 1.}(1)$.
Here, $j \colon C \rar \pr 1.$ is the function that detects the $j$ invariant of the smooth fibers \cite[Section IV.4]{Har77}.
Therefore, one would like to extend the work of Kodaira to the more general setup of a klt pair $f \colon (X,B) \rar Z$ with $\K X. + B \sim_\Q f^*L_Z$.
In particular, we are interested in writing $L_Z \sim_\Q \K Z. + B_Z + M_Z$, where $B_Z$ detects the singular fibers of $f$ and $M_Z$ detects the variation of the general fibers. 

Given a log canonical pair $(Y,\Gamma)$ and a $\Q$-Cartier divisor $D \geq 0$ on $Y$, we can measure ``how much of $D$'' we can add to $\Gamma$ while still preserving the log canonical property.
More precisely, we define the \emph{log canonical threshold} of $(Y,\Gamma)$ with respect to $D$ as
\[
\mathrm{lct}(Y,\Gamma;D) = \sup \lbrace t \geq 0 | (Y,\Gamma+tD) \; \mathrm{is} \; \mathrm{log} \; \mathrm{canonical} \rbrace.
\]
Since for some $c > 0$ we have $cD \geq \mathrm{Supp}(D)$, it follows that $(Y,\Gamma+c'D)$ is not log canonical for any $c' > c$.
In particular, $\mathrm{lct}(Y,\Gamma;D)$ is a well defined non-negative real number.
It turns out that, in the setup of a minimal elliptic surface $g \colon S \rar C$, we have $\mu_P (B_C) = 1 - \mathrm{lct}(S,0;g^*P)$ for every closed point $P \in C$.
In particular, under this perspective, Kodaira's algorithm to compute $B_C$ can be generalized to higher dimension.

Let $f \colon (X,B) \rar Z$ be a fibration as we considered above.
Now, we are ready to define a divisor $B_Z$ that generalizes the properties of the divisor $B_C$ computed by Kodaira.
For every prime divisor $P \subset Z$, we set the coefficient of $P$ in $B_Z$ as $\mu _P(B_Z) = 1 - \mathrm{lct}_{\eta_P}(X,B;f^*P)$.
Here $\mathrm{lct}_{\eta_P}(X,B;f^*P)$ denotes the log canonical threshold of $(X,B)$ with respect to $f^*P$ over the generic point of $P$.
The reason for this localization is twofold.
First, $P$ may not be $\Q$-Cartier, but it is at $\eta_P$ as $Z$ is normal.
Second, in this way we try to detect singularities that come from $(X,B)$, disregarding the ones coming from $P$.
Since $(X,B)$ is klt, $B_Z$ is a well defined Weil divisor.
By construction, it detects the singularities of the fibration over points of codimension 1 on the base.
Then, we can define $M_Z \coloneqq L_Z - (K_Z + B_Z)$.
Thus, we have $\K X. + B \sim_\Q f^*(\K Z. + B_Z + M_Z)$.

\begin{example}{\em
Let $X = \pr 1. \times \pr 2.$ and let $B$ be the disjoint union of two sections of $f \colon X \rar \pr 2.$.
By construction, $f \colon (X,B) \rar \pr 2.$ is an lc-trivial fibration (see the definition below).
Let $D \subset \pr 2.$ be a planar cuspidal cubic.
Then, we have $\mathrm{lct}(\pr 2.,0;D) = \frac{5}{6}.$
One can show that this implies that $\mathrm{lct}(X,B;f^*D)=\frac{5}{6}$.
This log canonical threshold is less than 1 because $f^*D$ is a $\pr 1.$-bundle over a cuspidal curve.
On the other hand, $f$ is smooth and so are its restrictions on the two sections.
Indeed, by inversion of adjunction we compute $\mathrm{lct} \subs \eta_D. (X,B;f^*D)=1$.
}
\end{example}

Let $\alpha \colon Z' \rar Z$ and $\beta \colon X' \rar X$ be projective birational morphisms.
Further, assume that the rational map $g \colon X' \drar Z'$ is a morphism.
Let $(X',B')$ be defined by $\K X'. + B' \coloneqq \beta^* (\K X. + B)$.
In general, the divisor $B'$ is not effective, but $\K X'. + B'$ shares many properties with the pair $(X,B)$.
We say that $(X',B')$ is a \emph{sub-pair}.
Since it is the pull-back of a klt pair, it is \emph{sub-klt}.
In particular, the log canonical threshold of $(X',B')$ with respect to an effective $\mathbb Q$-Cartier divisor is still well defined.
Thus, we can define a divisor $B \subs Z'.$ on $Z'$ as follows.
For every prime divisor $P' \subset Z'$, we have $\mu_{P'} (B \subs Z'.) = 1 - \mathrm{lct} \subs \eta_{P'}.(X',B',g^* P')$, where $\eta \subs P'.$ denotes the generic point of $P'$.
Then, we set $M \subs Z'. \coloneqq L \subs Z'. -(\K Z'. + B \subs Z'.)$.
By construction, we have $B_Z = \alpha_* B \subs Z'.$ and $M_Z = \alpha_* M \subs Z'.$.
In particular, b-$\mathbb{Q}$-divisors $\mathbf{B}_Z$ and $\mathbf{M}_Z$ are defined.
We refer to \cite{Cor07} for the notion of b-divisor.
We say that $\mathbf{B}_Z$ is the \emph{boundary b-divisor}, while $\mathbf{M}_Z$ is the \emph{moduli b-divisor}.
While the b-divisor $\mathbf B _Z$ is defined to detect geometric properties of the fibration $f \colon (X,B) \rar Z$, it is unclear whether $\mathbf M _Z$ has any interesting properties.

An \emph{lc-trivial fibration} $f \colon (X,B) \rar Z$ is a projective morphism with connected fibers between normal varieties such that
\begin{itemize}
    \item[(i)] $(X,B)$ is a sub-pair with coefficients in $\Q$ that is sub-log canonical over the generic point of $Z$;
    \item[(ii)] $\rk f_* \O X. (\lceil \mathbf{A}^*(X,B)\rceil)=1$; and
    \item[(iii)] there exists a $\Q$-Cartier divisor $L_Z$ on $Z$ such that $\K X. + B \sim_\Q f^* L_Z$.
\end{itemize}
We refer to \cite{FG14} for the definitions involved in the notion of lc-trivial fibration.
For the purposes of this note, it suffices to notice that the above conditions (i) and (ii) are satisfied if $(X,B)$ is a klt projective pair.
In the case that a fibration with log Calabi--Yau fibers is an lc-trivial fibration, the \emph{canonical bundle formula} describes how $\mathbf M _Z$ detects the variation of the fibers of the morphism $f$.
This is why $\mathbf{M}_Z$ is called moduli b-divisor.
The following formulation of the canonical bundle formula is \cite[Theorem 3.6]{FG14}.

\begin{canonical bundle} 
Let $f \colon (X,B) \rar Z$ be an lc-trivial fibration and let $\pi \colon Z \rar S$ be a projective morphism.
Let $\mathbf{B}_Z$ and $\mathbf{M}_Z$ be the b-divisors induced on $Z$.
Then, the b-divisor $\mathbf K _Z + \mathbf B _Z$ is b-$\Q$-Cartier.
Furthermore, the b-divisor $\mathbf{M}_Z$ is b-nef over $S$.
\end{canonical bundle}

\begin{remark}\label{remark cbf}
{\em
Since the statement of the canonical bundle formula involves the language of b-divisors, we rephrase its meaning in the case of a morphism between projective varieties.
In particular, we assume $S = \mathrm{Spec}(\C)$.
Under these assumptions, the content of the theorem is equivalent to the following.
There exists a birational morphism $\alpha \colon Z' \rar Z$ such that the divisor $M \subs Z'.$ is nef.
Furthermore, for any birational morphism $\gamma \colon Z'' \rar Z$ factoring through $Z'$ as $\sigma = \alpha \circ \rho$, we have $M \subs Z''. = \rho^* M \subs Z'.$ and $\K Z''. +  B \subs Z''. = \rho^*(\K Z'. + B \subs Z'.)$.
In particular, the birational model $g \colon (X',B') \rar Z'$ of the fibration $f \colon (X,B) \rar Z$ encodes all the information about all possible birational models of it.
Furthermore, the fact that $M \subs Z'.$ is nef should be thought as a weak analog of the fact that $M_C = \frac{1}{12}j^* \mathcal{O}_{\pr 1.}(1)$ in the case of an elliptic surface.
Indeed, thanks to work of Ambro and a subsequent generalization of Fujino and Gongyo \cites{Amb05,FG14}, something more is known about $M \subs Z'.$.
More precisely, under some technical assumptions, $M \subs Z'.$ is the pull-back of a nef and big divisor on a variety $T$.
Furthermore, $\dim (T)$ gives a Hodge-theoretic measure of the variation of the general fibers of $f$.
Given these positivity properties of $M \subs Z'.$, one can find $0 \leq D_Z \sim_\Q M_Z$ such that $(Z,\Delta_Z \coloneqq B_Z + D_Z)$ has mild singularities.
In particular, if $(X,B)$ is klt, then so is $(Z,\Delta_Z)$.
}
\end{remark}

\subsection*{Generalized pairs}
The canonical bundle formula is a very powerful tool in the study of lc-trivial fibrations.
For this reason, Birkar and Zhang defined an abstract object that encodes the properties of the outcome of the canonical bundle formula \cite{BZ}.

\begin{defn}
{\em A {\em generalized sub-pair} $(Z' \rar Z,B_Z,M\subs Z'.)$ is the datum of 
\begin{itemize}
    \item a normal variety $Z$; 
    \item a birational morphism $\alpha \colon Z' \rar Z$, where $Z'$ is normal; 
    \item a $\mathbb{Q}$-Weil divisor $B_Z$ on $Z$; and
    \item a $\Q$-Cartier divisor $M \subs Z'.$ on $Z'$ that is nef
\end{itemize}
such that $\K Z. + B_Z + M_Z$ is $\Q$-Cartier, where $M_Z \coloneqq \pi_* M \subs Z'.$. 
If $B_Z \geq 0$, we say $(Z' \rar Z,B_Z,M \subs Z'.)$ is a {\em generalized pair}.}
\end{defn}

In the case of an Iitaka fibration $f \colon X \rar Z$, $Z$ can be regarded as a generalized pair $(Z' \rar Z,B_Z,M \subs Z'.)$ with $\K Z. + B_Z + M_Z$ ample and $f^*(\K Z. + B_Z + M_Z) \sim_\Q \K X'.$.
Therefore, to discuss boundedness properties of varieties of intermediate Kodaira dimension, it is important to first address the boundedness of generalized pairs of general type.

Work of Birkar and Zhang shows that, together with $\dim (Z)$ and $\coeff(B_Z)$, one should fix the Cartier index of $M \subs Z'.$ in order to have control of the linear series $|m(\K Z. + B_Z + M_Z)|$.
Thus, fixing $n,r \in \N$, $v > 0$ and a DCC set $I \subset [0,1] \cap \Q$, it is interesting to investigate boundedness properties of generalized pairs $(Z' \rar Z,B_Z,M \subs Z'.)$ with $\dim (Z)=n$, $\vol(\K Z. + B_Z + M_Z)=v$, $\coeff(B_Z) \subset I$ and $rM \subs Z'.$ Cartier.
In this direction, \cite[Theorem 1.3]{BZ} implies that these generalized pairs are birationally bounded as pairs.
More precisely, the set consisting of $(Z,\Supp(B_Z))$ is birationally bounded.
Hence, it is natural to ask whether one can obtain honest boundedness if we further assume that $\K Z. + B_Z + M_Z$ is ample.

Surprisingly, this seems to be a hard question, as several technical difficulties come in the picture.
Since we assume that $B_Z$ has DCC coefficients, there is $\delta > 0$ such that $\delta \cdot \Supp(B_Z) \leq B_Z \leq \Supp(B_Z)$.
Thus, as in the work of Hacon, M\textsuperscript{c}Kernan and Xu, one can use intersection theoretic methods to bound $\Supp(B_Z)$.
On the other hand, $M_Z$ is just a pseudo-effective divisor, and in general, it is unclear how to bound it.
A possible approach is the following.
If we have $0 \leq H_Z \sim_\Q \K Z. + B_Z + M_Z$ with $\coeff(H_Z)$ bounded away from 0, one can bound $H_Z$.
Then, $M_Z$ is bounded up to $\Q$-linear equivalence as $H_Z - (\K Z. + B_Z)$.
An approach of this flavor is carried out in \cite{Fil18} in the case $\dim (Z) =2$, but it seems harder in general.

The second and subtler problem is the following.
Even assuming that we can choose a representative of $M_Z$ in its $\Q$-linear equivalence in order to guarantee that $(Z,B_Z+\Supp(M_Z))$ is birationally bounded, we still have no control of $M \subs Z'.$.
We illustrate this issue with the following example.
Assume that we have a set of generalized pairs $\lbrace (Z'_i \rar Z_i,B_{Z_i},M \subs Z'_i.) \rbrace \subs i \in I.$ and a projective morphism of quasi-projective varieties $(\mathcal{Z},\mathcal{D}) \rar T$ such that the following holds:
for every $(Z'_i \rar Z_i,B \subs Z_i.,M \subs Z'_i.)$ there exist a closed point $t(i) \in T$ and a birational rational map $f_i \colon Z_i \drar \mathcal{Z}_{t(i)}$ such that $\Supp(f_{i,*}B\subs Z_i.)\cup \Supp(f_{i,*}M\subs Z_i.) \cup {\rm Ex}(f_i \sups -1.) \subset \mathcal{D}_{t(i)}$.
In this situation, we may hope to find divisors $\mathcal{B}$ and $\mathcal{M}$ supported on $\mathcal{D}$ such that $f_{i,*}B\subs Z_i.=\mathcal{B}_{t(i)}$ and $f_{i,*}M\subs Z_i.=\mathcal{M}_{t(i)}$ for all $i \in I$.
Even if this is the case, we are still far away from being able to run the last part of the strategy in \cite{HMX18}, namely running a relative Minimal Model Program over $T$ and applying deformation invariance of plurigenera.
In order to apply a similar argument, we would need a condition close to the following:
there exist a birational morphism $\pi \colon \mathcal{Z}' \rar \mathcal{Z}$ and a divisor $\mathcal{M}'$ that is nef over $T$ such that $\mathcal{M}= \pi_* \mathcal{M}'$ and $\mathcal{M}'_{t(i)}$ is crepant to $M \subs Z'_i.$ for all $i \in I$.
This latter setup seems very hard to achieve in general, as given a generalized pair $(Z' \rar Z,B_Z,M \subs Z'.)$ it is hard to characterize how to optimally choose $Z'$ and how many blow-ups over $Z$ are required for such optimal choice.
In this direction, there are partial results just in dimension 2 \cite{Fil18}.

\section{Deformation invariance of plurigenera for generalized pairs} \label{def.inv.sect}

In this section, we focus on one of the steps that could possibly lead to boundedness for generalized pairs.
In birational geometry, one of the main invariants of a pair $(X,B)$ is its \emph{pluricanonical ring} $R(X,\K X. + B) \coloneqq \bigoplus \subs m \geq 0. H^0(X,m (\K X.+B))$.
Since in general $B$ is a fractional divisor, we define $H^0(X,m(\K X. + B)) \coloneqq H^0(X,\O X. (m \K X. + \lfloor m B \rfloor))$, so that $R(X,\K X. + B)$ has the structure of a graded ring.
By work of Birkar, Cascini, Hacon and M\textsuperscript{c}Kernan, we know that $R(X,\K X. + B)$ is finitely generated if $(X,B)$ is a projective klt pair \cite[Corollary 1.1.2]{BCHM}.
In particular, this guarantees that, if $\K X. + B$ is big, $\mathrm{Proj}(R(X,\K X. + B))$ recovers the \emph{canonical model} of $X$.

Therefore, when we have a family of pairs $(\mathcal{X},\mathcal{B}) \rar T$, it is natural to ask how the \emph{plurigenera} $h^0(\mathcal{X}_t,m(\K \mathcal{X}_t. + \mathcal{B}_t))$ behave as $t \in T$ varies.
A deep theorem, originally due to Siu \cite{Siu98}, states that the plurigenera are deformation invariant under mild assumptions.
For the reader's convenience, we include a version due to Hacon, M\textsuperscript{c}Kernan, and Xu that deals with the case of pairs \cite[Theorem 4.2]{HMX13}.
\begin{definv}
Let $\Xf \rar T$ be a flat projective morphism of quasi-projective varieties.
Let $(\Xf,\Delta)$ be a pair such that the fibers $(\Xf_t,\Delta_t)$ are $\Q$-factorial terminal for all $t \in T$.
Assume that every component $\mathcal{P}$ of $\Delta$ dominates $T$ and that the fibers of the Stein factorization of $\mathcal{P} \rightarrow T$ are irreducible.
Let $m > 1$ be any integer such that $\Df\coloneqq m(\K \Xf. +\Delta)$ is integral.
\newline
If either $\K \Xf. +\Delta$ or $\Delta$ is big over $T$, then $h^0(\Xf_t, \O X_t . (\Df_t))$ is independent of $t \in T$.
\end{definv}

Deformation invariance of plurigenera is a very important tool in proving the boundedness of pairs of general type.
Let $\lbrace (X_i,\Delta_i) \rbrace \subs i \in I.$ be a set of pairs with $\K X_i. + \Delta_i$ ample with fixed volume $v$ for all $i \in I$.
Assume that this set is log birationally bounded, and let $(\mathcal{X},\mathcal{B}) \rar T$ be a birationally bounding family.
Let $(\Xf_i,\Bf_i)$ denote the fiber corresponding to $(X_i,\Delta_i)$.
In order to obtain a bounding family for $\lbrace (X_i,\Delta_i) \rbrace \subs i \in I.$ from $(\mathcal{X},\mathcal{B}) \rar T$, we need to have $R(X_i,\K X_i . + \Delta_i) = R(\Xf_i,\K \Xf_i. + \Bf_i)$ for all $i$.
If that is the case, the relative canonical model of $(\mathcal{X},\mathcal{B}) \rar T$ will provide the needed family.
It is easy to show that we can guarantee $R(X_{i_0},\K X_{i_0} . + \Delta_{i_0}) = R(\Xf_{i_0},\K \Xf_{i_0}. + \Bf_{i_0})$ for a distinguished $i_0 \in I$.
By deformation invariance of plurigenera, one can show that the needed equality is satisfied for all $i \in I$.
This strategy is worked out in \cite[Proposition 7.3]{HMX18}.

In the hope that a similar strategy as above could be carried out in the setup of generalized pairs, we prove a version of deformation invariance of plurigenera for generalized pairs.
We follow the statement and proof of \cite[Theorem 4.2]{HMX13}.


\begin{theorem} \label{deformation invariance}
Let $\Xf \rar T$ be a flat projective morphism of quasi-projective varieties.
Let $(\Xf,\Delta)$ be a pair such that the fibers $(\Xf_t,\Delta_t)$ are $\Q$-factorial terminal for all $t \in T$.
Assume that every component $\mathcal{P}$ of $\Delta$ dominates $T$ and that the fibers of the Stein factorization of $\mathcal{P} \rightarrow T$ are irreducible.
Let $\Mf$ be a $\Q$-Cartier divisor that is nef over $T$.
Let $m>1$ be any integer such that $\Df\coloneqq m(\K \Xf. +\Delta +\Mf)$ is integral.
\newline
If either $\K \Xf. +\Delta + \Mf$ or $\Delta + \Mf$ is big over $T$, then $h^0(\Xf_t, \O X_t . (\Df_t))$ is independent of $t \in T$.
\end{theorem}

\begin{proof}
By the proof of \cite[Theorem 4.2]{HMX13}, we may assume that $T$ is a smooth affine curve and that $\Xf$ is $\Q$-factorial.
Furthermore, it is enough to show $|\Df_0| = |\Df|_{\Xf_0}$ for a special point $0 \in T$. By \cite[Lemma 4.4.(2)]{BZ}, the divisor $N_\sigma (\Xf_0, \K \Xf_0 . + \Delta_0 + \Mf_0)$ is a $\Q$-divisor. Therefore,
\[
\Theta_0 \coloneqq \Delta_0 - \Delta_0 \wedge N_\sigma (\Xf_0, \K \Xf_0 . + \Delta_0 + \Mf_0)
\]
is a $\Q$-divisor. Here $\wedge$ denotes the minimum between two divisor, taken prime component by prime component, while we refer to \cite[Definition-Lemma 3.3.1]{BCHM} for the definition of $N_\sigma$.
By assumption, there exists $0 \leq \Theta \leq \Delta$ whose restriction to $\Xf_0$ is $\Theta_0$.
Define
\[
\mu \coloneqq \frac{m}{m-1}.
\]
Then, the divisor $\K \Xf. + \mu (\Theta + \Mf)$ is big.
Therefore, we can find effective $\Q$-divisors $\Af$ and $\Bf$ such that $\Af$ is ample, $\Xf_0$ is not a component of $\Bf$, and $\K \Xf . + \mu(\Delta + \Mf) \sim_\Q \Af + \Bf$.
Up to shrinking $T$, we may assume that every irreducible component of $\Bf$ dominates $T$.
\newline
Now, we are going to perturbe the coefficients of $\Delta$ in order to apply Kawamata--Viehweg vanishing. Let $0 < \delta < \frac{1}{2}$ be a rational number, and define divisors
\[
\Xi \coloneqq (m-1-\delta)(\Delta - \Theta), \quad \Phi \coloneqq (1-\delta \mu + \delta)\Delta.
\]
Then, we can write
\begin{equation} \label{long equation}
\begin{split}
\Df - \Xi &= m(\K \Xf. + \Delta + \Mf) - (m-1-\delta)(\Delta - \Theta)\\
&= (m-1-\delta) (\K \Xf. + \Theta + \Mf) + (1+\delta)(\K \Xf. + \Delta + \Mf)\\
&= (m-1-\delta) (\K \Xf. + \Theta + \Mf) + \K \Xf. + \Delta + \Mf +\\ &\phantom{==} \delta(\K \Xf. + \mu(\Delta + \Mf)) - \delta (\mu - 1)(\Delta + \Mf)\\
& \sim_\Q  (m-1-\delta) (\K \Xf. + \Theta + \Mf) + \K \Xf. + \Delta + \Mf +\\ &\phantom{==} \delta \Af + \delta \Bf - \delta (\mu - 1)(\Delta + \Mf)\\
&= \K \Xf. + \Phi + (1-\delta \mu + \delta) \Mf +  \delta \Af + \delta \Bf + (m-1-\delta) (\K \Xf. + \Theta + \Mf)\\
& \sim_\Q  \K \Xf . + \Phi + \Hf + \delta \Bf + (m-1-\delta)(\K \Xf. + \Theta + \Hf'),
\end{split}
\end{equation}
where
\[
\Hf \sim_\Q (1- \delta \mu + \delta)\Mf + \frac{\delta}{2}\Af, \quad \Hf' \sim_\Q \Mf + \frac{\delta}{2(m-1-\delta)}\Af.
\]
Since $\Mf$ is nef, then $\Hf$ and $\Hf'$ are ample.
Therefore, we may assume that their supports are irreducible and general depending on $\delta$.
Thus, if $\delta$ is small enough, $(\Xf_t, \Delta_t + \Hf_t + \Hf'_t + \delta \Bf_t)$ is terminal for every $t \in T$.
\newline
Since $\Hf'_0$ is ample, no component of $\Theta_0 + \Hf'_0$ belongs to the stable base locus of $\K \Xf_0 . + \Theta_0 + \Hf'_0$.
Hence, by \cite[Proposition 4.1]{HMX13}, there exists a log terminal model $f \colon \Xf \dashrightarrow \Yf$ for $(\Xf,\Theta + \Hf')$ over $T$ that induces a weak log canonical model $f_0 \colon \Xf_0 \dashrightarrow \Yf_0$ of $(\Xf_0,\Theta_0 + \Hf'_0)$.
\newline
Let $p \colon \Wf \rar \Xf$ and $q \colon \Wf \rar \Yf$ resolve $f$. We may also assume that $p$ is a log resolution for $(\Xf,\Delta + \Mf + \Hf + \Hf' + \Bf)$.
Define
\[
G\coloneqq (m-1-\delta)f_* (\K \Xf. + \Theta + \Hf').
\]
Thus, $G$ is nef and big and we have
$$
(m-1-\delta) p^*(\K \Xf. + \Theta + \Hf') = q^*G + F,
$$
where $F$ is effective and $q$-exceptional.
\newline
Let $\Wf_0$ be the strict transform of $\Xf_0$.
Since $(\Xf_0,\Phi_0+ \delta \Bf_0 + \Hf_0)$ is klt, by inversion of adjunction \cite[Theorem 5.50]{KM}, $(\Xf,\Xf_0 + \Phi + \delta \Bf + \Hf)$ is plt. 
Therefore, we can write
\[
\K \Wf. + \Wf_0 = p^*(\K \Xf. + \Xf_0 + \Phi + \delta \Bf + \Hf) + E,
\]
where $\lceil E \rceil \geq 0$ is $p$-exceptional.
Now, set
\[
L\coloneqq \lceil p^*(\Df - \Xi) +E -F \rceil.
\]
Since we may assume $\Xf_0 \sim_\Q 0$, we can write
\[
\K \Wf. + \Wf_0 \sim_\Q p^*(\K \Xf. + \Phi + \delta \Bf + \Hf) + E.
\]
By \eqref{long equation}, we have
\[
p^*(\Df - \Xi) \sim_\Q q^*G +F + p^*(\K \Xf. + \Phi +  \delta \Bf + \Hf).
\]
Thus, we have
\[
\K \Wf. + q^*G \sim_\Q p^*(\Df - \Xi) +E -F - \Wf_0.
\]
This implies that
\[
L - \Wf_0 \sim_\Q \K \Wf. + C +q^*G,
\]
where $C$ is the fractional part of $-p^*(\Df - \Xi) -E + F$.
Since $C$ is supported on divisors involved in the log resolution, and its coefficients are less than 1, $(\Wf, C)$ is klt.
Therefore, Kawamata--Vieheweg vanishing implies
\[
H^1(\Wf,\O \Wf.(L-\Wf_0))=0.
\]
\newline
Let $N\coloneqq p^*(\K \Xf. +\Theta + \Mf) - q^*f_*(\K \Xf. +\Theta + \Mf)$.
Since 
\[
Q \coloneqq (\K \Xf. +\Theta + \Hf')-(\K \Xf. +\Theta + \Mf) \sim_\Q \frac{\delta}{2(m-1-\delta)}\Af
\]
is ample, the negativity lemma implies $p^* Q \leq q^*f_* Q$ \cite[Lemma 3.39]{KM}.
Therefore, we have
\[
mN = (1 + \delta) N + (m-1-\delta)N \geq F.
\]
Since $\Xi \leq m(\Delta - \Theta)$, it follows $\Df - \Xi \geq m(\K \Xf. + \Theta + \Mf)$.
Hence, we can write
\begin{equation*}
\begin{split}
R \coloneqq{}& L - \lfloor mq^*f_*(\K \Xf. + \Theta + \Mf) \rfloor \\
={}& \lceil L - mq^*f_*(\K \Xf. + \Theta + \Mf) \rceil \\
\geq{}& \lceil mN +E - F \rceil \\
\geq{}& \lceil E \rceil .
\end{split}
\end{equation*}
Let $q_0 \colon \Wf_0 \rar \Xf_0$ be the restriction of $q$, and denote by $L_0$ and $R_0$ the restrictions of $L$ and $R$ to $\Wf_0$ respectively. Then, we have
\begin{align*}
|\Df_0| &= |m(\K \Xf_0 . + \Theta_0 + \Mf_0)| & \text{by definition of } \Theta_0 \\
&\subset |mf_{0,*}(\K \Xf_0 . + \Theta_0 + \Mf_0)| & \text{as } f_0 \text{ is a birational contraction} \\
&= |mq_0^*f_{0,*}(\K \Xf_0 . + \Theta_0 + \Mf_0)| & \\
&\subset |L_0| & \text{as } R_0 \geq 0\\
&=|L|_{\Wf_0} & \text{as } H^1(\Wf,\O \Wf.(L-\Wf_0))=0\\
&\subset |\Df|_{\Xf_0} & \text{as } \lceil E \rceil \text{ is } p\text{-exceptional}.
\end{align*}
As the reversed inclusion $|\Df|_{\Xf_0} \subset |\Df_0|$ always holds, we conclude that we have the equality $|\Df|_{\Xf_0}=|\Df_0|$.
\end{proof}

\begin{remark}{\em
The above result can be generalized to the case when $\Mf$ is nef over some points $\lbrace t_i \rbrace_{i \geq 1} \subset T$: it would follow that $h^0(\Xf_i,\O \Xf_i. (\Df_{t_i}))$ is independent of $i$.
The proof is a slight modification of the above one: we have to compare $t_1$ and $t_2$ pairwise. Hence, we can base change to a smooth curve containing both of them. Then, we perform the constructions as above in order to satisfy the required properties over $t_1$ and $t_2$. Then, by openness, we can shrink the base so that the properties in the proof of Theorem \ref{deformation invariance} are satisfied. This setup is more technical, yet very useful: nefness is neither open nor closed in families.
}
\end{remark}

\section{An example of boundedness for fibrations}
\label{fibr.sect}
In this section, we show an example of how boundedness statements can be proven inductively, in the case of fibrations.
Here, we shall focus on the case of a variety $X$ endowed with an elliptic fibration $f \colon X \rar Y$, i.e., a morphism whose general fiber is a smooth elliptic curve.
Moreover, we shall assume that the variety $X$ is a minimal model with Kodaira dimension $\kappa(X)=\dim(X)-1$, in the sense that $\K X.$ is nef and 
\begin{align}
\label{eqn:kod.dim}
    h^0(X,m\K X.) \sim C m^{\dim(X)-1}+o(m^{\dim(X)-1})
\end{align} for $m$ large and divisible.
This implies that $X$ is the outcome of a run of the Minimal Model Program for a projective variety $X'$ with mild singularities -- see the statement of Theorem \ref{thm boundedness modulo flops} for the precise assumptions on singularities;
furthermore, $f \colon X \rar Y$ is the Iitaka fibration of $X$,~\cite[\S 2.1.C]{LAZ1}.
Then $\K X. \sim_\Q f^*L$ for some big and nef $\Q$-Cartier $\Q$-divisor $L$ on $Y$. 
By the $(n-1)$-volume of $\K X.$, we shall mean $\vol(Y,L)$; this is independent of the choice of the divisor $L$ in its $\mathbb Q$-linear equivalence class.

While it is well known that the fibers of $f$ vary in a bounded family, in order to understand properties of the base $Y$, a natural choice is to use the canonical bundle formula, see, for example, \cites{Amb04, Amb05}.
A priori, this gives a structure of generalized pair $(Y' \rar Y,B,M')$ on the base, and we do not have boundedness statements for generalized pairs in dimension greater than 2.
Fortunately, this problem can be circumvented as follows.
Since we are considering fibrations of relative dimension 1, a particular case of a conjecture due to Prokhorov and Shokurov implies that we can turn $M$ into an effective divisor $\Delta$ and have control of its coefficients \cite[Conjecture 7.13, Theorem 8.1]{PS09}.
Therefore, the base is endowed with a structure of a pair $(Y,\Gamma)$, and we control the coefficients of $\Gamma$.
Then, when the fibration corresponds to the Iitaka fibration, the pair $(Y,\Gamma)$ is of general type.
Thus, at least at the birational level, to understand the structure of $X$, we can first address the boundedness of the pair $(Y,\Gamma)$.

This ``divide and rule'' approach, together with some geometric assumptions on the fibration $f \colon X \rar Y$, leads to a particular form of weak boundedness.
We say that a set of varieties $\lbrace X_i \rbrace \subs i \in I.$ is \emph{bounded in codimension 1} if there is a projective morphism $\Xf \rar T$ of schemes of finite type such that every $X_i$ is isomorphic in codimension 1 to $\Xf_{t(i)}$ for some closed point $t(i) \in T$.
Furthermore, if $\Xf _{t(i)}$ is normal and projective with $\K \Xf_{t(i)}.$ $\Q$-Cartier for all $i \in I$, we say that $\lbrace X_i \rbrace \subs i \in I.$ is \emph{bounded modulo flops}.
Notice that these notions are stronger than the usual birational boundedness, as the failure of honest boundedness happens in codimension 2.

\begin{theorem} \label{thm boundedness modulo flops}
Fix a positive integer $n$, and a positive real number $v$.
Then the set $\mathfrak{D}(n,v)$ of varieties $X$ such that
\begin{enumerate}
    \item $X$ is a terminal projective variety of dimension $n$,
    \item $X$ is minimal of Kodaira dimension $n-1$,
    \item the $(n-1)$-volume of $K_X$ is $v$, and
    \item the Iitaka fibration $f \colon X \rar Y$ of $X$ admits a rational section,
\end{enumerate}
is bounded modulo flops.
\end{theorem}

\begin{proof}
By assumption, we have $\kappa(\K X.)=n-1$, and, by \cite[Remark 1.2]{Leh13}, we have $\nu (\K X.)=\kappa (\K X.)$.
Thus, $\K X.$ is nef and abundant.
By \cite[Theorem 4.3]{GL13}, $X$ admits a good minimal model, so that, by \cite[Proposition 2.4]{Lai11}, $\K X.$ is semi-ample.
Therefore, the Iitaka fibration $f \colon X \drar Y$ is a morphism, and $Y$ is a normal $(n-1)$-fold.
Furthermore, by \cite[Lemma 5.1]{dCS17}, the rational section of $f$ is defined over a big open set of $Y$.
We now divide the proof into several steps.
\newline
{\bf Step 0:} {\it In this step we show that, if $X \in \mathfrak{D}(n,v)$, then $Y$ belongs to a bounded family only depending on $\mathfrak D (n,v)$.}
\newline
As $Y$ is the projective variety associated to the ring of sections of the canonical bundle of $X$, by \cite{FM00}, $Y$ is endowed with a natural structure of generalized klt generalized pair $(Y' \rar Y,B,M')$ and $\vol(Y,\K Y. + B + M)=v$.
As the generic fiber of $f \colon X \rar Y$ is an elliptic curve, by \cite{FM00}, the coefficients of $B$ are in a fixed DCC set $\Lambda \subset [0,1) \cap \Q$ independent of $X \in \mathfrak{D}(n,v)$.
Similarly, the Cartier index of $M'$ is independent of $X \in \mathfrak{D}(n,v)$.
Furthermore, there exists an integer $k$ depending just on $n$ such that $|kM'|$ is a free linear series, cf. \cite[Theorem 8.1]{PS09}.
Therefore, we can choose a general element $0 \leq H' \sim_\Q M'$ such that $kA'$ is a prime divisor and $(Y,B+A)$ is klt, where $A$ is the push-forward of $A'$ onto $Y$.
Thus, $(Y,B+A)$ is a klt pair of dimension $n-1$, $\coeff(B+A) \subset \Lambda \cup \lbrace \frac{1}{k} \rbrace$ and $\vol(Y,\K Y. + B + A)=v$.
By Theorem \ref{boundedness hmx}, $(Y,B+A)$ belongs to a bounded family of pairs depending just on $n, v$.
\newline
{\bf Step 1:} {\it In this step we reduce to the case when $X$ and $Y$ are $\Q$-factorial
and $f$ has a rational section which is well defined over a big open set of Y
.
}
\newline
Since $(Y,B+A)$ is klt, $Y$ admits a small $\Q$-factorialization $(Y'',B''+A'') \rar (Y,B+A)$.
By \cite{1610.08932}, also $Y''$ belongs to a bounded family which only depends on $n$ and $v$.
Let $\pi \colon X'' \rar X$ be a smooth resolution of indeterminacies for the map $X \drar Y''$.
As $X \rar Y$ is a fibration in curves, no exceptional divisor of $X'' \rar X$ dominates $Y''$.
Let $E$ denote the reduced $\pi$-exceptional divisor.
Since $X$ is terminal, we have $\K X''. + \frac{1}{2} E= \pi^* \K X. + F$, where $F \geq 0$ is supported on all of the $\pi$-exceptional divisors.
We have $\K X''. + \frac{1}{2} E \sim \subs \Q,Y''. F$. 
Now, we can run a $(\K X''. + \frac{1}{2} E)$-MMP with scaling relative to $Y''$.
As the image of $F$ on $Y''$ has codimension at least 2, $F$ is degenerate in the sense of \cite[Definition 2.8]{Lai11}.
Thus, \cite[Lemma 2.9]{Lai11} implies that this MMP terminates with a model $X''' \rar Y''$ on which the sitrct transform $F'''$ of $F$ is 0.
Thus, we have that $X'''$ is $\Q$-factorial and $\K X'''. \sim \subs \Q,Y''. 0$.
Furthermore, as we contracted all the $\pi$-exceptional divisors, $X''' \drar X$ is small.
\newline
Since $X''' \drar X$ is an isomorphism in codimension 1, it suffices to show that $X'''$ is bounded.
By construction, we have that $X''' \rar Y''$ admits a rational section.
Over the big open set of $Y''$ where $Y'' \rar Y$ is an isomorphism, $X$ and $X'''$ differ by flops over $Y$.
Thus, the rational section of $X'''\rar Y''$ is a section over a big open set of $Y''$. 
Hence, up to relabelling and assuming that $f$ has a section just over a big open of $Y$, we may assume that $X=X'''$ and $Y=Y''$.
\newline
{\bf Step 2:} {\it In this step we find a birational model of $X$ where the rational section satisfies certain positivity assumptions.}
\newline
Now, denote by $\hat Y$ the closure of the rational section of $f \colon X \rar Y$.
Then, $\hat Y$ is relatively big over $Y$.
Also, for $0 < \gamma \ll 1$, $(X,\gamma \hat Y)$ is klt.
Thus, by \cite{BCHM}, any $(\K X. + \gamma \hat Y)$-MMP over $Y$ with scaling of an ample divisor terminates.
Let $(\tilde X,\gamma \tilde Y)$ be the resulting model.
Denote by $\tilde f \colon \tilde X \rar Y$ the resulting morphism.
Notice that $\K X. + \gamma \hat Y \sim \subs \Q,Y. \gamma \hat Y$.
Thus, this MMP is independent of $\gamma$, and $\tilde Y$ is relatively big and semi-ample over $Y$.
Furthermore, since $\hat Y$ is irreducible and dominates $Y$, every step of the above MMP has to be a $(\K X. + \gamma \hat Y)$-flip.
Thus, $\tilde X$ is isomorphic to $X$ in codimension 1 and it suffices to prove that $\tilde X$ is bounded.
Moreover, as $\K X. \sim \subs \Q,Y. 0$, the terminality of $X$ implies that of $\tilde X$.
Thus, $X$ and $\tilde X$ differ by a sequence of $K_X$-flops.
\newline
{\bf Step 3:} {\it In this step we show that $(\tilde X,\tilde Y)$ is plt pair. This implies that $\tilde Y$ is a normal $\mathbb{Q}$-Gorenstein variety and that the pair $(\tilde Y, 0)$ is klt.}
\newline
Normality of $\tilde Y$ and kltness of $(\tilde Y, 0)$ will follow from the pltness of $(\tilde X,\tilde Y)$ by \cite[Proposition 5.51]{KM} and inversion of adjunction, see \cite{Kaw07}.
To show that $(\tilde X,\tilde Y)$ is plt, it suffices to show that $(\tilde X, \tilde Y)$ is log canonical and that $\tilde Y$ is its only log canonical center.
Let $\phi \colon \tilde Y ^\nu \rar \tilde Y$ be the normalization of $\tilde Y$, and let $\mathrm{Diff}(0)$ be the different defined by
\[
\K \tilde Y ^\nu. + \mathrm{Diff}(0) \coloneqq \phi^*((\K \tilde X. + \tilde Y)| \subs \tilde Y.).
\]
By construction, $\K \tilde Y ^\nu. + \mathrm{Diff}(0)$ is nef and big over $Y$.
By \cite[Lemma 5.1]{dCS17}, $\mathrm{Diff}(0)$ is exceptional over $Y$.
Thus, we have $(\tilde f \circ \phi)_*(\K \tilde Y ^\nu. + \mathrm{Diff}(0))=\K Y.$.
Since $Y$ is $\Q$-factorial, the negativity lemma \cite[Lemma 3.39]{KM} implies that
\begin{align} 
\label{eqtn_klt}
\K \tilde Y ^\nu. + \mathrm{Diff}(0) = (\tilde f \circ \phi)^* \K Y. - D, 
\end{align}
where $D \geq 0$ is $(\tilde f \circ \phi)$-excetpional.
As $(Y,B+A)$ is klt, then so is $(Y,0)$.
Therefore, it follows from \eqref{eqtn_klt} that $(\tilde Y ^\nu,\mathrm{Diff}(0))$ is klt.
Inversion of adjunction implies that $\tilde Y$ is the only log canonical center of $(\tilde X, \tilde Y)$.
In particular, $(\tilde X, \tilde Y)$ is plt and the other conclusions follow as indicated above.
\newline
{\bf Step 4:} {\it In this step 
we show that there exists an effective divisor $\tilde G$ on $\tilde X$ such that the pair $(\tilde X, \frac 1 2 \tilde Y + \frac 1 2 \tilde G)$ is $\frac 1 2$-klt, and $K_{\tilde X}+\frac 1 2 \tilde Y + \frac 1 2 \tilde G$ is big.}
%
\newline
Let $H$ be a very ample polarization on $Y$ whose existence is guaranteed by the boundedness of the pairs $(Y,B+A)$, cf. \S \ref{sect.bound}.
Moreover, by definition of boundedness, there exists a positive real number $C=C(n, v)$ such that $\vol(Y, H) \leq C$.
Let $\tilde{G}$ be a general member of $|(2n+2) \tilde f ^*H|$.
Then, $(\tilde X,\tilde{Y}+\tilde G)$ is log canonical.
On the other hand, $\tilde X$ is terminal, and the discrepancies of valuations are linear functions of the boundary divisor of a pair.
Hence, it follows that $(\tilde X, \frac{1}{2} \tilde Y + \frac{1}{2} \tilde G)$ is $\frac{1}{2}$-klt.
Since $\K \tilde X.$ is the pull-back of a big and nef divisor on $Y$, $\tilde Y$ is effective and relatively big over $Y$, it follows that $\K \tilde X. + \frac{1}{2}\tilde Y + \frac{1}{2} \tilde G$ is big.
Since we have
\[
\K \tilde X. + \frac{1}{2}\tilde Y + \frac{1}{2} \tilde G = \frac{1}{2} \K \tilde X. + \frac{1}{2} (\K \tilde X. + \tilde Y + \tilde G),
\]
and $\K \tilde X.$ is nef, it suffices to show that $\K \tilde X. + \tilde Y + \tilde G$ is nef to conclude that so is $\K \tilde X. + \frac{1}{2}\tilde Y + \frac{1}{2} \tilde G$.
Nefness of $\K \tilde X. + \tilde Y + \tilde G$ follows by the boundedness of the negative extremal rays \cite[Theorem 1.19]{Fuj14}.
Indeed, let $R$ be a $(\K \tilde X. + \tilde Y)$-negative extremal ray.
There exists a rational curve $C$ spanning $R$ such that $-2n \leq (\K \tilde X. + \tilde Y)\cdot C < 0$.
Since $\K \tilde X. + \tilde Y$ is nef relatively to $Y$, then $\tilde f (C)$ is a curve.
In particular, we have $G \cdot C \geq (2n+2)H \cdot \tilde f (C) \geq 2n+2$.
So, it follows that $\K \tilde X. + \tilde Y + \tilde G$ is non-negative on every $(\K \tilde X. + \tilde Y)$-negative extremal ray.
Thus, $\K \tilde X. + \tilde Y + \tilde G$ is nef.
In particular, we have that $\K \tilde X. + \frac{1}{2}\tilde Y + \frac{1}{2} \tilde G$ is nef and big.
\newline
{\bf Step 5:} {\it In this step we show that there exist positive constants $C_1$ and $C_2$, only depending on $n$ and $v$, such that $C_1 \leq (\K \tilde X. + \frac{1}{2}\tilde Y + \frac{1}{2} \tilde G)^n \leq C_2$.}
\newline
The existence of $C_1$ follows from \cite[Theorem 1.3]{HMX14b}.
Thus, we are left to show the existence of $C_2$.
Now, by the differentiability of the volume function \cite{LM09}, we have
\[
\left. \frac{d}{dt} \right| \subs t=s. \vol(\K\tilde X. + t\tilde{Y} + \tilde{G}) = \vol\subs \tilde X|\tilde{Y}. (\K \tilde X. + s \tilde{Y} + \tilde{G}),
\]
where $\vol \subs \tilde{X}|\tilde{Y}.$ denotes the restricted volume function \cite{MR2530849}.
Furthermore,
\begin{align*}
\begin{split}
\K \tilde{X}. + s\tilde{Y}&= s(\K \tilde{X}. + \tilde{Y}) + (1-s) \K \tilde{X}.\\&=s(\K \tilde{X}. + \tilde{Y}) + (1-s)(\tilde{f})^*(\K Y. + B + A). 
\end{split}
\end{align*}
Thus,
\begin{align*}
\begin{split}
\vol\subs \tilde{X}|\tilde{Y}. (\K \tilde{X}. + s \tilde{Y} + \tilde{G})  &\leq \vol ((\K \tilde{X}. + s \tilde{Y} + \tilde{G})| \subs \tilde{Y}.))\\
&= \vol (s(\K \tilde{Y}. + \Diff(0))+(\tilde{g})^*((1-s)(\K Y. + B + A)+(2n+2)H) \\
&=\vol (\K Y. + (1-s)(\K Y. + B + A)+(2n+2)H) \leq C_2,
\end{split}
\end{align*}
where we set $\tilde{g} \colon \tilde{Y} \rar Y$, and $C_2$ only depends on $n,v$.
Then, we conclude that
\[
\vol(\K \tilde{X}. + \tilde{Y} + \tilde{G}) = \int_0^1 \left. \frac{d}{dt} \right| \subs t=s. \vol(\K \tilde{X}. + t \tilde{Y} + \tilde{G})ds \leq C_2.
\]
Since,
\[
\vol\left(\K \tilde X. + \frac{1}{2} \tilde Y + \frac{1}{2} \tilde G\right) \leq \vol (\K \tilde{X}. + \tilde{Y} + \tilde{G}),
\]
the claim follows.
\newline
{\bf Step 6:} {\it In this step we conclude the proof.}
\newline
As showed in the previous steps, $(\tilde X, \frac{1}{2}\tilde Y + \frac{1}{2} \tilde G)$ is $\frac{1}{2}$-klt and its coefficients belong to the set $\lbrace \frac{1}{2} \rbrace$.
Thus, by \cite[Theorem 1.3]{Fil18b}, $\vol\left(\K \tilde X. + \frac{1}{2} \tilde Y + \frac{1}{2} \tilde G\right)$ belongs to a discrete set only depending on $n$ and $v$.
By Step 5, this volume is also bounded from above and below.
Thus, we conclude that $\vol\left(\K \tilde X. + \frac{1}{2} \tilde Y + \frac{1}{2} \tilde G\right)$ attains only finitely many values, only depending on $n$ and $v$.
Then, by \cite[Theorem 6]{1610.08932}, the set of pairs $(\tilde X, \frac{1}{2}\tilde Y + \frac{1}{2} \tilde G)$ is log bounded.
In particular, the varieties $\tilde{X}$ are bounded.
This concludes the proof.
\end{proof}

\subsection*{Kawamata--Morrison conjecture and boundedness} 
The statement of Theorem \ref{thm boundedness modulo flops} provides evidence that the ``divide and rule'' approach can lead to boundedness statements for Calabi--Yau fibrations.
Given a fibration $f \colon X \rar Y$ as in the statement of Theorem \ref{thm boundedness modulo flops}, and letting $g \colon X' \rar Y'$ be the model constructed in the proof that belongs to a bounded family, then $Y'$ is a small $\Q$-factorialization of $Y$.
For simplicity, we assume that $Y=Y'$ and $X$ is $\Q$-factorial.
Thus, both $X$ and $X'$ are two minimal models for $f \colon X \rar Y$.
Let $\alpha \colon X \drar X'$ be the induced rational map.
Since $\alpha$ is an isomorphism in codimension 1, the divisors on $X$ and $X'$ are naturally identified.
In particular, we get an isomorphism $\alpha_* \colon N^1(X/Y) \rar N^1(X'/Y)$ between the vector spaces of $\R$-Cartier divisors modulo numerical equivalence over $Y$.
Under this morphism, we get identifications $\alpha_* \overline{\mathrm{Eff}}(X/Y)=\overline{\mathrm{Eff}}(X'/Y)$ and $\alpha_* \overline{\mathrm{Mov}}(X/Y)=\overline{\mathrm{Mov}}(X'/Y)$ between the closures of the relative cones of effective and movable divisors respectively \cite{Kaw97}.
On the other hand, $\alpha_* \mathrm{Nef}(X/Y)$ is not in general mapped to $\mathrm{Nef}(X'/Y)$, unless $\alpha$ is an isomorphism.
More precisely, we have that either $\alpha_* \mathrm{Nef}(X/Y)=\mathrm{Nef}(X'/Y)$, or $\alpha_* \mathrm{Int}(\mathrm{Nef}(X/Y)) \cap \mathrm{Int}(\mathrm{Nef}(X'/Y))= \emptyset$, where ${\rm Int}$ indicates the interior of a set; the first case occurs if and only if $\alpha$ is an isomorphism \cite[Lemma 1.5]{Kaw97}.
Then, $\overline{\mathrm{Mov}}(X/Y)$ can be decomposed into chambers, each one corresponding to $\mathrm{Nef}(X'/Y)$ for some model $X'$ isomorphic to $X$ in codimension 1.
Therefore, to study all the possible minimal models of $f \colon X \rar Y$ we should analyze the cones ${\mathrm{Mov}}(X/Y)$ and $\mathrm{Nef}(X/Y)$.
It could happen that a minimal model $X'$ is isomorphic to $X$, while the rational map over $Y, \; \alpha \colon X \drar X'$ is not an isomorphism, cf. \cite[Example 3.8.(2)]{Kaw97}.
Thus, we may have more chambers corresponding to the same isomorphism class of varieties.

In the setup of Theorem \ref{thm boundedness modulo flops}, a first step towards proving the boundedness of the initial input $f \colon X \rar Y$ would be to show that there are just finitely many relative minimal models $g \colon X' \rar Y$.
This is exactly the content of the Kawamata--Morrison cone conjecture.
\begin{conconj}[Kawamata--Morrison]
\cite[Conjecture 2.1]{Tot10}
\label{km.conj}
Let $f \colon X \rar Y$ be a projective morphism with connected fibers between normal varieties. Let $(X,\Delta)$ be a klt pair such that $\K X. + \Delta \equiv 0/Y$. Also, define ${\rm Nef}^e(X/Y) \coloneqq {\rm Nef}(X/Y) \cap {\rm Eff}(X/Y)$ and ${\rm Mov}^e(X/Y) \coloneqq {\rm Mov}(X/Y) \cap {\rm Eff}(X/Y)$. Then, the following holds.
\begin{itemize}
    \item[1] The number of ${\rm Aut}(X/Y,\Delta)$-equivalence classes of faces of the cone ${\rm Nef}^e(X/Y)$ corresponding to birational contractions or fiber space structures is finite. Moreover, there exists a rational polyhedral cone $\Pi$ which is a fundamental domain for the action of ${\rm Aut}(X/Y,\Delta)$ on ${\rm Nef}^e(X/Y)$ in the sense that
    \begin{itemize}
        \item [a] ${\rm Nef}^e(X/Y)=\bigcup \subs g \in {\rm Aut}(X/Y,\Delta). g_* \Pi$; and
        \item[b] ${\rm Int}\Pi \cap g_*{\rm Int}\Pi = \empty$ unless $g_*=1$.
    \end{itemize}
    \item[2] The number of ${\rm PsAut}(X/Y,\Delta)$-equivalence classes of chambers ${\rm Nef}^e(X'/Y)$ in ${\rm Mov}^e(X/Y)$ corresponding to marked small $\Q$-factorial modifications $X' \rar Y$ of $X \rar Y$ is finite. Equivalently, the number of isomorphism classes over $Y$ of small $\Q$-factorial modifications of $X$ over $Y$ (ignoring the birational identification with $X$) is finite. Moreover, there exists a rational polyhedral cone $\Pi'$ which is a fundamental domain for the action of ${\rm PsAut}(X/Y,\Delta)$ on ${\rm Mov}^e(X/Y)$.
\end{itemize}
\end{conconj}
This is a very deep conjecture connecting the birational geometry of a log Calabi--Yau fibration to the structure of the (birational) automorphism group.
The intuition behind such connection is rooted in mirror symmetry and physics, see, for example, \cite{MR1265317}, but it is still unclear how exactly to determine the existence of automorphism starting from the geometry of the cone of divisors.
Conjecture \ref{km.conj} is known to hold just in very few cases: Totaro proved it in dimension 2, \cite{Tot10}, Kawamata proved the relative case for threefolds without boundary, \cite{Kaw97}, and there are a few other cases known in dimension $>2$.

Assuming the Kawamata--Morrison cone conjecture, one could hope to explore the following approach in order to improve the statement of Theorem \ref{thm boundedness modulo flops} to actual boundedness.
First, one would need to show that the number of models of $f$ connected by relative flops is bounded in a family and provides a constructible function on the base.
Once this is settled, in order to achieve boundedness, one would need to argue that one can extend flops from a general fiber to an open set over the base.
If this were the case, by finitely many flops of the birationally bounding family one would recover a bounding family for the initial moduli problem, as shown in \cite{1610.08932} for the log general type case.


\bibliographystyle{alpha}
\bibliography{bib} 

\begin{bibdiv}
\begin{biblist}

\bib{Amb04}{article}{
      author={Ambro, F.},
       title={Shokurov's boundary property},
        date={2004},
        ISSN={0022-040X},
     journal={J. Differential Geom.},
      volume={67},
      number={2},
       pages={229\ndash 255},
      review={\MR{2153078}},
}

\bib{Amb05}{article}{
      author={Ambro, F.},
       title={The moduli {$b$}-divisor of an lc-trivial fibration},
        date={2005},
        ISSN={0010-437X},
     journal={Compos. Math.},
      volume={141},
      number={2},
       pages={385\ndash 403},
      review={\MR{2134273}},
}

\bib{BCHM}{article}{
      author={Brikar, C.},
      author={Cascini, P.},
      author={Hacon, C.},
      author={McKernan, J.},
       title={Existence of minimal models for varieties of log general type},
        date={2010},
      volume={23},
      number={2},
       pages={405\ndash 468},
}

\bib{BdCS}{misc}{
      author={Birkar, C.},
      author={Di~Cerbo, G.},
      author={Svaldi, R.},
       title={Birational boundedness of elliptic calabi-yau varieties with a
  section},
        date={2019},
        note={In preparation},
}

\bib{Bir16b}{misc}{
      author={Birkar, C.},
       title={{Singularities of linear systems and boundedness of Fano
  varieties}},
        date={2016},
        note={arxiv e-print
  \href{https://arxiv.org/abs/1609.05543}{1609.05543v1}},
}

\bib{Bir16a}{article}{
      author={Birkar, C.},
       title={Anti-pluricanonical systems on {F}ano varieties},
        date={2019},
        ISSN={0003-486X},
     journal={Ann. of Math. (2)},
      volume={190},
      number={2},
       pages={345\ndash 463},
      review={\MR{3997127}},
}

\bib{BZ}{article}{
      author={Birkar, C.},
      author={Zhang, D.-Q.},
       title={Effectivity of {I}itaka fibrations and pluricanonical systems of
  polarized pairs},
        date={2016},
        ISSN={0073-8301},
     journal={Publ. Math. Inst. Hautes \'{E}tudes Sci.},
      volume={123},
       pages={283\ndash 331},
      review={\MR{3502099}},
}

\bib{CDHJS}{misc}{
      author={Chen, W.},
      author={Di~Cerbo, G.},
      author={Han, J.},
      author={Jiang, C.},
      author={Svaldi, R.},
       title={Birational boundedness of rationally connected {C}alabi--{Y}au
  3-folds},
        date={2018},
        note={arXiv e-print,
  \href{https://arxiv.org/abs/1804.09127}{1804.09127v1}},
}

\bib{Cor07}{incollection}{
      author={Corti, A.},
       title={Introduction},
        date={2007},
   booktitle={Flips for 3-folds and 4-folds},
      series={Oxford Lecture Ser. Math. Appl.},
      volume={35},
   publisher={Oxford Univ. Press, Oxford},
       pages={1\ndash 17},
      review={\MR{2359339}},
}

\bib{dCS17}{misc}{
      author={Di~Cerbo, G.},
      author={Svaldi, R.},
       title={Birational boundedness of low dimensional elliptic
  {C}alabi--{Y}au varieties with a section},
        date={2017},
        note={arXiv e-print
  \href{https://arxiv.org/abs/1608.02997}{1608.02997v3}},
}

\bib{MR0262240}{article}{
      author={Deligne, P.},
      author={Mumford, D.},
       title={The irreducibility of the space of curves of given genus},
        date={1969},
        ISSN={0073-8301},
     journal={Inst. Hautes \'{E}tudes Sci. Publ. Math.},
      number={36},
       pages={75\ndash 109},
      review={\MR{0262240}},
}

\bib{MR2530849}{article}{
      author={Ein, L.},
      author={Lazarsfeld, R.},
      author={Musta\c{t}\u{a}, M.},
      author={Nakamaye, M.},
      author={Popa, M.},
       title={Restricted volumes and base loci of linear series},
        date={2009},
        ISSN={0002-9327},
     journal={Amer. J. Math.},
      volume={131},
      number={3},
       pages={607\ndash 651},
      review={\MR{2530849}},
}

\bib{FG14}{article}{
      author={Fujino, O.},
      author={Gongyo, Y.},
       title={On the moduli b-divisors of lc-trivial fibrations},
        date={2014},
        ISSN={0373-0956},
     journal={Ann. Inst. Fourier (Grenoble)},
      volume={64},
      number={4},
       pages={1721\ndash 1735},
      review={\MR{3329677}},
}

\bib{Fil18}{article}{
      author={Filipazzi, S.},
       title={Boundedness of log canonical surface generalized polarized
  pairs},
        date={2018},
        ISSN={1027-5487},
     journal={Taiwanese J. Math.},
      volume={22},
      number={4},
       pages={813\ndash 850},
      review={\MR{3830822}},
}

\bib{Fil18b}{article}{
      author={Filipazzi, S.},
       title={Some remarks on the volume of log varieties},
        date={2020},
        ISSN={0013-0915},
     journal={Proc. Edinb. Math. Soc. (2)},
      volume={63},
      number={2},
       pages={314\ndash 322},
      review={\MR{4089377}},
}

\bib{FM00}{article}{
      author={Fujino, O.},
      author={Mori, S.},
       title={A canonical bundle formula},
        date={2000},
        ISSN={0022-040X},
     journal={J. Differential Geom.},
      volume={56},
      number={1},
       pages={167\ndash 188},
      review={\MR{1863025}},
}

\bib{Fuj14}{article}{
      author={Fujino, O.},
       title={Fundamental theorems for semi log canonical pairs},
        date={2014},
        ISSN={2214-2584},
     journal={Algebr. Geom.},
      volume={1},
      number={2},
       pages={194\ndash 228},
      review={\MR{3238112}},
}

\bib{GL13}{article}{
      author={Gongyo, Y.},
      author={Lehmann, B.},
       title={Reduction maps and minimal model theory},
        date={2013},
        ISSN={0010-437X},
     journal={Compos. Math.},
      volume={149},
      number={2},
       pages={295\ndash 308},
      review={\MR{3020310}},
}

\bib{Har77}{book}{
      author={Hartshorne, R.},
      editor={Springer-Verlag},
       title={Algebraic geometry},
      series={Graduate Texts in Mathematics},
        date={1977},
      volume={52},
}

\bib{HM06}{article}{
      author={Hacon, C.~D.},
      author={McKernan, J.},
       title={Boundedness of pluricanonical maps of varieties of general type},
        date={2006},
        ISSN={0020-9910},
     journal={Invent. Math.},
      volume={166},
      number={1},
       pages={1\ndash 25},
      review={\MR{2242631}},
}

\bib{HM10}{inproceedings}{
      author={Hacon, C.~D.},
      author={McKernan, J.},
       title={Flips and flops},
        date={2010},
   booktitle={Proceedings of the {I}nternational {C}ongress of
  {M}athematicians. {V}olume {II}},
   publisher={Hindustan Book Agency, New Delhi},
       pages={513\ndash 539},
      review={\MR{2827807}},
}

\bib{HMX13}{article}{
      author={Hacon, C.~D.},
      author={McKernan, J.},
      author={Xu, C.},
       title={On the birational automorphisms of varieties of general type},
        date={2013},
        ISSN={0003-486X},
     journal={Ann. of Math. (2)},
      volume={177},
      number={3},
       pages={1077\ndash 1111},
      review={\MR{3034294}},
}

\bib{HMX14b}{article}{
      author={Hacon, C.~D.},
      author={McKernan, J.},
      author={Xu, C.},
       title={A{CC} for log canonical thresholds},
        date={2014},
        ISSN={0003-486X},
     journal={Ann. of Math. (2)},
      volume={180},
      number={2},
       pages={523\ndash 571},
      review={\MR{3224718}},
}

\bib{HMX18}{article}{
      author={Hacon, C.~D.},
      author={McKernan, J.},
      author={Xu, C.},
       title={Boundedness of moduli of varieties of general type},
        date={2018},
        ISSN={1435-9855},
     journal={J. Eur. Math. Soc. (JEMS)},
      volume={20},
      number={4},
       pages={865\ndash 901},
      review={\MR{3779687}},
}

\bib{Kaw07}{article}{
      author={Kawakita, M.},
       title={Inversion of adjunction on log canonicity},
        date={2007},
        ISSN={0020-9910},
     journal={Invent. Math.},
      volume={167},
      number={1},
       pages={129\ndash 133},
      review={\MR{2264806}},
}

\bib{Kaw97}{article}{
      author={Kawamata, Y.},
       title={On the cone of divisors of {C}alabi-{Y}au fiber spaces},
        date={1997},
        ISSN={0129-167X},
     journal={Internat. J. Math.},
      volume={8},
      number={5},
       pages={665\ndash 687},
      review={\MR{1468356}},
}

\bib{KM}{book}{
      author={Koll\'{a}r, J.},
      author={Mori, S.},
       title={Birational geometry of algebraic varieties},
      series={Cambridge Tracts in Mathematics},
   publisher={Cambridge University Press, Cambridge},
        date={1998},
      volume={134},
        ISBN={0-521-63277-3},
        note={With the collaboration of C. H. Clemens and A. Corti, Translated
  from the 1998 Japanese original},
      review={\MR{1658959}},
}

\bib{Kod63}{article}{
      author={Kodaira, K.},
       title={On compact analytic surfaces. {II}, {III}},
        date={1963},
        ISSN={0003-486X},
     journal={Ann. of Math. (2) 77 (1963), 563--626; ibid.},
      volume={78},
       pages={1\ndash 40},
      review={\MR{0184257}},
}

\bib{koll.mod}{incollection}{
      author={Koll\'{a}r, J.},
       title={Moduli of varieties of general type},
        date={2013},
   booktitle={Handbook of moduli. {V}ol. {II}},
      series={Adv. Lect. Math. (ALM)},
      volume={25},
   publisher={Int. Press, Somerville, MA},
       pages={131\ndash 157},
      review={\MR{3184176}},
}

\bib{Kol13}{book}{
      author={Koll\'{a}r, J.},
      editor={Press, Cambridge~University},
       title={Singularities of the minimal model program},
      series={Cambridge Tracts in Mathematics},
        date={2013},
      volume={200},
        note={With the collaboration of S\'{a}ndor J. Kov\'{a}cs},
}

\bib{k.book96}{book}{
      author={Koll\'{a}r, J.},
       title={Rational curves on algebraic varieties},
      series={Ergebnisse der Mathematik und ihrer Grenzgebiete. 3. Folge. A
  Series of Modern Surveys in Mathematics [Results in Mathematics and Related
  Areas. 3rd Series. A Series of Modern Surveys in Mathematics]},
   publisher={Springer-Verlag, Berlin},
        date={1996},
      volume={32},
        ISBN={3-540-60168-6},
      review={\MR{1440180}},
}

\bib{Lai11}{article}{
      author={Lai, C.-J.},
       title={Varieties fibered by good minimal models},
        date={2011},
        ISSN={0025-5831},
     journal={Math. Ann.},
      volume={350},
      number={3},
       pages={533\ndash 547},
      review={\MR{2805635}},
}

\bib{LAZ1}{book}{
      author={Lazarsfeld, R.},
      editor={Springer-Verlag},
       title={Positivity in algebraic geometry, i},
      series={Ergebnisse der Mathematik},
        date={2004},
      volume={48},
}

\bib{Leh13}{article}{
      author={Lehmann, B.},
       title={Comparing numerical dimensions},
        date={2013},
        ISSN={1937-0652},
     journal={Algebra Number Theory},
      volume={7},
      number={5},
       pages={1065\ndash 1100},
      review={\MR{3101072}},
}

\bib{LM09}{article}{
      author={Lazarsfeld, R.},
      author={Musta\c{t}\u{a}, M.},
       title={Convex bodies associated to linear series},
        date={2009},
        ISSN={0012-9593},
     journal={Ann. Sci. \'{E}c. Norm. Sup\'{e}r. (4)},
      volume={42},
      number={5},
       pages={783\ndash 835},
      review={\MR{2571958}},
}

\bib{MR1265317}{article}{
      author={Morrison, D.~R.},
       title={Compactifications of moduli spaces inspired by mirror symmetry},
        date={1993},
        ISSN={0303-1179},
     journal={Ast\'{e}risque},
      number={218},
       pages={243\ndash 271},
        note={Journ\'{e}es de G\'{e}om\'{e}trie Alg\'{e}brique d'Orsay (Orsay,
  1992)},
      review={\MR{1265317}},
}

\bib{1610.08932}{misc}{
      author={Martinelli, D.},
      author={Schreieder, S.},
      author={Tasin, L.},
       title={On the number and boundedness of log minimal models of general
  type},
        date={2016},
        note={arXiv e-print
  \href{https://arxiv.org/abs/arXiv:1610.08932}{arXiv:1610.08932v5}. Accepted
  for publication in {\it Annales scientifiques de l'\'Ecole normale
  sup\'erieure}.},
}

\bib{PS09}{article}{
      author={Prokhorov, Y.~G.},
      author={Shokurov, V.~V.},
       title={Towards the second main theorem on complements},
        date={2009},
        ISSN={1056-3911},
     journal={J. Algebraic Geom.},
      volume={18},
      number={1},
       pages={151\ndash 199},
      review={\MR{2448282}},
}

\bib{Siu98}{article}{
      author={Siu, Y.-T.},
       title={Invariance of plurigenera},
        date={1998},
        ISSN={0020-9910},
     journal={Invent. Math.},
      volume={134},
      number={3},
       pages={661\ndash 673},
      review={\MR{1660941}},
}

\bib{Tot10}{article}{
      author={Totaro, B.},
       title={The cone conjecture for {C}alabi-{Y}au pairs in dimension 2},
        date={2010},
        ISSN={0012-7094},
     journal={Duke Math. J.},
      volume={154},
      number={2},
       pages={241\ndash 263},
      review={\MR{2682184}},
}

\bib{Voisin}{book}{
      author={Voisin, C.},
      editor={Press, Cambridge~University},
       title={Hodge theory and complex algebraic geometry i},
        date={2002},
}

\end{biblist}
\end{bibdiv}

\end{document}